\documentclass[eqnumis, pdf]{yjb}
\usepackage{amsfonts}
\usepackage{amsmath}
\usepackage{footnote}
\usepackage{multicol}
\usepackage{booktabs}
\usepackage[flushleft]{threeparttable}
\usepackage[font=small,labelfont=bf]{caption}
\usepackage{float}
\usepackage{hyperref}
\DeclareMathOperator{\sech}{sech}
\hypersetup{
colorlinks,
citecolor=red,
linkcolor=blue,
urlcolor=pink}
\usepackage{mathrsfs,upref}

\begin{document}

%

\markboth{Yogesh J. Bagul, Ramkrishna M. Dhaigude, Barkat A. Bhayo}{Wilker and Huygens type inequalities for mixed trigonometric-hyperbolic functions}

\title{Wilker and Huygens type inequalities for mixed trigonometric-hyperbolic functions}

\author{Yogesh J. Bagul $^{1}$, Ramkrishna M. Dhaigude$^{2}$, Barkat A. Bhayo\coraut$^{3}$, Vinay M. Raut$^{4}$ }

\address{ $^{1}$Department of Mathematics,\\ K. K. M. College, Manwath,\\ Dist: Parbhani(M. S.) - 431505, India\\
$^{2}$Department of Mathematics,\\ Government Vidarbha Institute of Science\\ and Humanities, Amravati(M. S.)-444604, India\\
$^{*3}$Department of Mathematics, Sukkur IBA University,\\
 Sindh, Pakistan\\
 $^{4}$Department of Mathematics, \\
Shri. Shivaji Science College, Amravati(M. S.)-444603, India\\
	}
\emails{yjbagul@gmail.com, rmdhaigude@gmail.com, barkat.bhayo@iba-suk.edu.pk, vinayraut18@gmail.com}

\maketitle
\begin{abstract}
Motivated by the work of S\'{a}ndor \cite{sandor},  
in this paper we  establish a new Wilker type and Huygens type inequalities involving the trigonometric and hyperbolic functions. Moreover, in terms of hyperbolic functions, the upper and lower bounds of $\sin(x)/x$ and $\tan(x)/x$ are given.  
\end{abstract}



\vspace{10pt}

\section{Introduction}\label{sec1}
In the last two decades, the refinements of the inequalities involving trigonometric and hyperbolic functions such as Wilker type inequalities and Huygens type inequalities have been studied by several authors e.g., see \cite{bercu, bhayo, chu, guo, klen,neuman1, sandor, wus,  zhang, zhu} and the references therein. Motivated by the work of S\'andor \cite{sandor}, and above studies, 
in this paper we make a contribution to the subject by establishing a new Wilker type and Huygens type inequalities involving the trigonometric and hyperbolic functions. In all cases, we
give the upper and lower bounds of $\sin(x)/x$ and $\tan(x)/x$ in terms of elementary functions.
        
For $0<x<\pi/2$, Wilker \cite{wilker} and Huygens \cite{huygens} proposed the following inequalities
\begin{align}\label{eqn1.1}
\left(\frac{\sin x}{x}\right)^2 + \frac{\tan x}{x} > 2,
\end{align}
\begin{align}\label{eqn1.5}
2 \frac{\sin x}{x} + \frac{\tan x}{x} > 3,
\end{align}
respectively. In literature, inequality (\ref{eqn1.1}) and (\ref{eqn1.5}) are known as 
 Wilker's and Huygens' inequalities, respectively.

In \cite{zhu}, Zhu proved the the hyperbolic version of (\ref{eqn1.1}) as follows, 
\begin{align}\label{eqn1.3}
\left(\frac{\sinh x}{x}\right)^2 + \frac{\tanh x}{x} > 2,\qquad x>0,
\end{align}
Similarly, Neuman and S\'andor \cite{neuman} established  
\begin{align}\label{eqn1.7}
2 \frac{\sinh x}{x} + \frac{\tanh x}{x} >  3 \quad x>0,
\end{align}
as a hyperbolic version of \eqref{eqn1.5}.

Recently, Wu and Srivastava \cite{wu} proved that for $0 < x < \pi/2$
\begin{align}\label{eqn1.2}
\left(\frac{x}{\sin x}\right)^2 + \frac{x}{\tan x} > 2,
\end{align}
and the following hyperbolic version of \eqref{eqn1.2}
\begin{align}\label{eqn1.4}
\left(\frac{x}{\sinh x}\right)^2 + \frac{x}{\tanh x} > 2 \quad (x > 0),
\end{align}
was proved by Wu and Debnath in \cite{wus}.

For $0<x<\pi/2$, the following relations were established by Neuman and S\'andor
\begin{align}\label{eqn1.6}
2 \frac{\sin x}{x} + \frac{\tan x}{x} > 2 \frac{x}{\sin x} + \frac{x}{\tan x} > 3,
\end{align}

\begin{align}\label{eqn1.7}
2 \frac{\sinh x}{x} + \frac{\tanh x}{x} > 2 \frac{x}{\sinh x} + \frac{x}{\tanh x} > 3,
\end{align}
in \cite{neuman} and \cite{neuman1}, respectively. For $0 < x < \pi/2$, one can write that  
\begin{align}\label{eqn1.8}
\frac{\sin x}{x} + 2 \frac{\tan x}{x} > 2 \frac{\sin x}{x} + \frac{\tan x}{x},
\end{align}
because $\cos x<1$.

For the extensions, generalizations, refinements and alternative proofs of the above inequalities, the reader is referred to see, e.g., \cite{bercu, bhayo, neuman, neuman1, wu, wus, zhu, chu, matejicka, zhu1, mhanna, zhang, rasajski, mortici, nantomah, guo}. Inequality (\ref{eqn1.1}) and (\ref{eqn1.5}) are known in the literature as Wilker's and Huygens inequalities respectively. Rest of the inequalities are known to be Wilker-type and Huygens-type inequalities for either trigonometric functions or hyperbolic functions. In contrast to this, S\'{a}ndor 
\cite{sandor} established the following inequalities for $0 < x < \pi/2$,
\begin{align}\label{eqn1.9}
\left(\frac{\sinh x}{x}\right)^2 + \frac{\sin x}{x} > 2,
\end{align}
\begin{align}\label{eqn1.10}
\left(\frac{x}{\sin x}\right)^2 + \frac{x}{\sinh x} > 2,
\end{align}
\begin{align}\label{eqn1.11}
2 \frac{\sinh x}{x} + \frac{\sin x}{x} > 3,
\end{align}
and
\begin{align}\label{eqn1.12}
2\frac{x}{\sin x} + \frac{x}{\sinh x} > 3.
\end{align}
In \cite{sandor}, the inequalities (\ref{eqn1.9}), (\ref{eqn1.10}) are called  as trigonometric-hyperbolic Wilker type inequalities and the inequalities (\ref{eqn1.11}), (\ref{eqn1.12}) are called as trigonometric-hyperbolic Huygens type inequalities. The main purpose of this paper is to obtain some trigonometric-hyperbolic Wilker type and Huygens type inequalities. 

The paper is organized as follows. In this section, we give an introduction and highlight the concerned previous inequalities
together with the statements of main results in the form of theorems. Section 2 consists of preliminary ancestor and lemmas, which will be used in the proving procedures sequel. Section 3 closes the paper by proving the main results from Section 1.


Theorems \ref{thm1}-\ref{thm2} and 
Theorems \ref{thm3}-\ref{thm4} are dealing with Wilker-type inequalities and Huygens-type inequalities, respectively. 

Our first main result reads as follows.

\begin{theorem}\label{thm1}
For $ 0 < x < \pi/2 $ it is asserted that
\begin{align}\label{eqn3.1}
2 - \frac{x^2}{2} < \frac{x}{\sin x} + \left(\frac{\tanh x}{x}\right)^2  < 2 
< 2 + \frac{x^4}{180} < \left(\frac{\sin x}{x}\right)^2 + \frac{x}{\tanh x}.
\end{align}
\end{theorem}

\begin{theorem}\label{thm2}
The inequalities
\begin{align}\label{eqn3.3}
\frac{\sinh x}{x} + \left(\frac{x}{\tan x}\right)^2 < 2 - \frac{x^2}{6} < 2 
< 2 + \frac{x^5}{45 \tan x} < \left(\frac{\sinh x}{x}\right)^2 + \frac{x}{\tan x}
\end{align}
hold for all $ x \in (0, \pi/2).$
\end{theorem}

\begin{theorem}\label{thm3}
Let $ 0 < x < \pi/2. $ Then the following inequalities hold:
\begin{align}\label{eqn3.9}
2 \frac{\sinh x}{x} + \frac{x}{\tan x} < 3 - \frac{x^4}{180} < 3 
< 3 +  \frac{31}{180}x^4 < 2 \frac{x}{\sinh x} + \frac{\tan x}{x}
\end{align}
and
\begin{align}\label{eqn3.10}
 \frac{\sinh x}{x} + 2 \frac{x}{\tan x} < 3 - \frac{x^2}{2} < 3 
< 3 + \frac{x^2}{2} <  \frac{x}{\sinh x} + 2 \frac{\tan x}{x}.
\end{align}
\end{theorem}

\begin{theorem}\label{thm4}
If $ 0 < x < \pi/2 $ then the following inequalities hold:
\begin{align}\label{eqn3.11}
3 - \frac{x^4}{180} < 2 \frac{\sin x}{x} + \frac{x}{\tanh x} < 3 
< 2 \frac{x}{\sin x} + \frac{\tanh x}{x} < 3 + \frac{31}{180}x^4,
\end{align}
and
\begin{align}\label{eqn3.12}
3 - \frac{x^2}{2} < \frac{x}{\sin x} + 2 \frac{\tanh x}{x} < 3 
< \frac{\sin x}{x} + 2 \frac{x}{\tanh x} < 3 + \frac{x^2}{2}.
\end{align}
\end{theorem}

\section{Preliminaries and Lemmas}\label{sec2}

In this section, we give a few series expansion formulas and lemmas, which will be used in the proofs of our main result.

For the following series expansions, we refer to \cite[1.411]{grad}.
\begin{align}\label{eqn2.1}
\cos x = \sum_{k=0}^{\infty}  \frac{(-1)^k}{(2k)!}x^{2k}, \ \ \cosh x = \sum_{k=0}^{\infty} \frac{x^{2k}}{(2k)!},
\end{align}
\begin{align}\label{eqn2.2}
\frac{\sin x}{x} = \sum_{k=0}^{\infty} \frac{(-1)^k}{(2k+1)!}x^{2k}, \ \ \frac{\sinh x}{x} = \sum_{k=0}^{\infty} \frac{x^{2k}}{(2k+1)!}, \ \ x \neq 0,
\end{align}
\begin{align}\label{eqn2.3}
\frac{x}{\sin x} =  1 + \sum_{k=1}^{\infty} \frac{2(2^{2k-1}-1)}{(2k)!}\vert B_{2k} \vert x^{2k} , \ \ \vert x \vert < \pi,
\end{align}

\begin{align}\label{eqn2.31}
\frac{x}{\sinh x} =  1 - \sum_{k=1}^{\infty} \frac{2(2^{2k-1}-1)}{(2k)!} B_{2k} x^{2k} , \ \ \vert x \vert < \pi,
\end{align}

\begin{align}\label{eqn2.4}
x \cot x = 1 - \sum_{k=1}^{\infty}\frac{2^{2k} \vert B_{2k} \vert}{(2k)!} x^{2k}, \ \ \vert x \vert < \pi, 
\end{align}
\begin{align}\label{eqn2.5}
x \coth x = 1 + \sum_{k=1}^{\infty}\frac{2^{2k}  B_{2k} }{(2k)!} x^{2k}, \ \ \vert x \vert < \pi, 
\end{align}  
\begin{equation}\label{eqn2.6}
\left.
\begin{split}
\tan x &= \sum_{k=1}^{\infty} \frac{2^{2k}(2^{2k}-1)}{(2k)!}\vert B_{2k} \vert x^{2k-1}, \ \ \vert x \vert < \pi/2 \\
&= x + \frac{x^3}{3} + \frac{2 x^5}{15} + \frac{17 x^7}{315} + \frac{62 x^9}{2835} + \cdots
\end{split}
\right\}
\end{equation}   

\begin{equation}\label{eqn2.7}
\left.
\begin{split}
\tanh x &= \sum_{k=1}^{\infty} \frac{2^{2k}(2^{2k}-1)}{(2k)!} B_{2k} x^{2k-1}, \ \ \vert x \vert < \pi/2 \\
&= x - \frac{x^3}{3} + \frac{2 x^5}{15} - \frac{17 x^7}{315} + \frac{62 x^9}{2835} - \cdots
\end{split}
\right\}
\end{equation} 
Differentiating (\ref{eqn2.7}) we get
$$ \sech^2 x = 1 - \tanh^2 x = \sum_{k=1}^{\infty} \frac{2^{2k}(2k-1)(2^{2k}-1)}{(2k)!} B_{2k} x^{2k-2}, \ \ \vert x \vert < \pi/2. $$
Hence
\begin{align}\label{eqn2.8}
\left(\frac{\tanh x}{x}\right)^2 =  - \sum_{k=2}^{\infty} \frac{2^{2k}(2k-1)(2^{2k}-1)}{(2k)!} B_{2k} x^{2k-4}, \ \ \vert x \vert < \pi/2.
\end{align}
The following lemmas are also necessary.

\begin{lemma}\label{lem1} \cite{klen}
For $ 0 < x < \pi/2, $ we have
\begin{align}\nonumber
1 - \frac{x^2}{6} < \frac{\sin x}{x} < \frac{x}{\sinh x}.
\end{align}
\end{lemma}
\begin{lemma}\label{lem2}\cite{bercu}
For every $ x \neq 0 , $ we have
\begin{align}\nonumber
1 + \frac{5 x^2}{x^2 + 15} < \frac{x}{\tanh x}.
\end{align}
\end{lemma}

\begin{lemma}\label{lem3}
For $ 0 < x < \pi/2, $ it is true that
\begin{align}\nonumber
\frac{x}{\tanh x} < \frac{\tan x}{x}.
\end{align}
\end{lemma}
\begin{proof}
The famous Adamovi\'{c}-Mitrinovi\'{c} inequality \cite{mitrinovic}
$$ \cos x < \left( \frac{\sin x}{x} \right)^3  \ \ (0 < x < \pi/2) $$ can be written as $$ \left(\frac{x}{\sin x} \right)^2 < \frac{\tan x}{x}  \ \ (0 < x < \pi/2). $$
Combining this with the inequality \cite[Proposition 1.2]{bagul}
$$ \frac{x}{\tanh x} < \left(\frac{x}{\sin x} \right)^2  \ \ (0 < x < \pi/2), $$ we get the required inequality.
\end{proof}

\begin{lemma}\label{lem4}\cite[p. 805]{abramowitz}
\begin{align}\nonumber
  \frac{2(2k)!}{(2 \pi)^{2k}} < \vert B_{2k} \vert < \frac{2(2k)!}{(2 \pi)^{2k}} \frac{2^{2k-1}}{2^{2k-1}-1}, \ \  (k=1, 2, 3, \cdots)
\end{align}
where $ B_2, B_4, B_6, \cdots $ are Bernoulli numbers.
\end{lemma}

\begin{lemma}\label{lem5}
 The inequality
\begin{align}\nonumber
\vert B_{2k} \vert > \frac{1}{2^{2k-1}(2k+1)} = \frac{2^{1-2k}}{(2k+1)}
\end{align}
holds for all integers $ k \geq 2.$
\end{lemma}
\begin{proof}
It is evident that
$$ (2k + 1)! > \pi^{2k} \ \ (k=2, 3, \cdots) $$
$$ \text{i.e.,} \ \ \frac{1}{\pi^{2k}} > \frac{1}{(2k+1)!}. $$ This yields
$$ \vert B_{2k} \vert > \frac{(2k)!}{2^{2k-1} \pi^{2k}} > \frac{(2k)!}{2^{2k-1} (2k+1)!} =  \frac{2^{1-2k}}{(2k+1)} $$ due to Lemma \ref{lem4}.
\end{proof}

\begin{lemma}\label{lem6}
The inequality
\begin{align}\nonumber
\vert B_{2k} \vert > \frac{1}{2^{2k}(2k+1)} = \frac{2^{-2k}}{(2k+1)}
\end{align}
holds for all integers $ k \geq 1. $
\end{lemma}
\begin{proof}
Since, $$ 2(2k+1)! > \pi^{2k} \ \ (k=1, 2, 3, \cdots),$$
we get the desired result by applying the same argument as in the proof of Lemma \ref{lem5}.
\end{proof}

\begin{lemma}\label{lem7}

The inequality
\begin{align}\nonumber
\frac{\vert B_{2k+2} \vert}{\vert B_{2k} \vert} > \frac{(k+1)(2k+1)(2^{2k-1} - 1)}{2^{2k-1}(2^{2k+2} - 1)}
\end{align}
is true for all integers $ k \geq 2. $
\end{lemma}
\begin{proof}
From Lemma \ref{lem4}, we can write
\begin{align*}
\frac{\vert B_{2k+2} \vert}{\vert B_{2k} \vert} &> \frac{2(2k+2)!}{(2\pi)^{2k+2}} \frac{(2\pi)^{2k}(2^{2k-1}-1)}{2(2k)! 2^{2k-1}} \\
&= \frac{(k+1)(2k+1)(2^{2k-1}-1)}{ 2^{2k}\pi^2} \\
&> \frac{(k+1)(2k+1)(2^{2k-1}-1)}{2^{2k-1}(2^{2k+2} - 1)},
\end{align*}
as $ 2^{2k} \pi^2 < 2^{2k-1}(2^{2k+2}-1) $ i.e., $ 2 \pi^2 + 1 < 2^{2k+2} $ for all integers $ k \geq 2. $
\end{proof}

\section{Proofs and corollaries}\label{sec3}

\noindent{\bf Proof of Theorem \ref{thm1}.}
For the leftmost double inequality of (\ref{eqn3.1}), we add (\ref{eqn2.3}) and (\ref{eqn2.8}) to get
\begin{align*}
\frac{x}{\sin x} + \left(\frac{\tanh x}{x}\right)^2 &= 1 + \sum_{k=1}^{\infty} \frac{2(2^{2k-1}-1)}{(2k)!}\vert B_{2k} \vert x^{2k} \\
&- \sum_{k=2}^{\infty} \frac{2^{2k}(2k-1)(2^{2k}-1)}{(2k)!} B_{2k} x^{2k-4}, \ \ \vert x \vert < \pi/2 \\
&= 2 + \sum_{k=1}^{\infty} \left[ \frac{2(2^{2k-1}-1)}{(2k)!} \vert B_{2k} \vert - \frac{2^{2k+4}(2k+3)(2^{2k+4}-1)}{(2k+4)!} B_{2k+4} \right] x^{2k} \\
&= 2 + \left( \sum_{k=1}^{\infty} a_k x^{2k} := A(x) \right), \ \ \vert x \vert < \pi/2
\end{align*}
 where $ a_k = \frac{2(2^{2k-1}-1)}{(2k)!} \vert B_{2k} \vert - \frac{2^{2k+4}(2k+3)(2^{2k+4}-1)}{(2k+4)!} B_{2k+4} \ \ (k=1, 2, 3, \cdots). $
 
Clearly  $ a_k > 0 $ for $ k = 2, 4, 6, \cdots. $
For $ k = 1, 3, 5, 7, \cdots $ we claim that $ a_k < 0 $ i.e.,
$$ \frac{\vert B_{2k} \vert}{\vert B_{2k+4} \vert} < \frac{2^{2k+1}(2^{2k+4}-1)}{(k+1)(k+2)(2k+1)} \frac{1}{(2^{2k-1}-1)} = L(k). $$
From Lemma \ref{lem4}, we write
$$ \vert B_{2k} \vert < \frac{2(2k)!}{2^{2k} \pi^{2k}} \frac{2^{2k-1}}{(2^{2k-1}-1)} \ \ \text{and} \ \ \frac{1}{\vert B_{2k+4} \vert} < \frac{2^{2k+4} \pi^{2k+4}}{2(2k+4)!}. $$ Therefore
$$ \frac{\vert B_{2k} \vert}{\vert B_{2k+4} \vert} < \frac{1}{(k+1)(k+2)(2k+1)} \frac{1}{((2^{2k-1}-1)} \frac{\pi^4 2^{2k+1}}{(2k+3)} = M(k). $$
Since $ \pi^4 < (2k+3)(2^{2k+4}-1) \ \ (k = 1, 2, 3, \cdots) $ implies $ M(k) < L(k), $ so $ a_k < 0 $ for $ k = 1, 3, 5, \cdots $ and we conclude that $ A(x) $ is an alternating convergent series whose first term is negative. The required double inequality follows by the truncation of $ A(x).$

For the rightmost double inequality of (\ref{eqn3.1}), we have by Lemmas \ref{lem1} and \ref{lem2} that 
\begin{align*}
\left(\frac{\sin x}{x}\right)^2 + \frac{x}{\tanh x} &> \left(1 - \frac{x^2}{6} \right)^2 + \left(1 + \frac{5x^2}{x^2 + 15}\right) \ \ (0 < x < \pi/2) \\
&= 2 + \frac{x^4}{36} - \frac{x^4}{3x^2 + 45} \\
&= 2 + \frac{x^4}{36} \left(\frac{x^2 + 3}{x^2 + 15} \right) \\
&> 2 + \frac{x^4}{36} \left(\frac{x^2 + 3}{5x^2 + 15} \right) = 2 + \frac{x^4}{180}  > 2. 
\end{align*}
This completes the proof.
$\hfill\square$

\begin{remark}
If $ n $ is any odd positive integer and $ 0 < x < \pi/2 $ then
\begin{align}\nonumber
2 + \sum_{k=1}^{n} a_k x^{2k} < \frac{x}{\sin x} + \left(\frac{\tanh x}{x}\right)^2 < 2 + \sum_{k=1}^{n+1} a_k x^{2k}
\end{align}
where $ a_k $ is as defined in the proof of Theorem \ref{thm1}. In particular, we get
$$ 2 - \frac{x^2}{2} < \frac{x}{\sin x} + \left(\frac{\tanh x}{x}\right)^2 < 2 - \frac{x^2}{2} + \frac{143 x^4}{360}. $$
\end{remark}

\vspace{.5cm}

\noindent{\bf Proof of Theorem \ref{thm2}.}
As $ \frac{x}{\tan x} < 1 \ \ (0 < x < \pi/2), $ we can write
$$ \frac{\sinh x}{x} + \left(\frac{x}{\tan x}\right)^2 < \frac{\sinh x}{x} + \frac{x}{\tan x} \ \ (0 < x < \pi/2). $$
Utilizing series expansions from (\ref{eqn2.2}) and (\ref{eqn2.4}),
\begin{align*}
\frac{\sinh x}{x} + \left(\frac{x}{\tan x}\right)^2 &< \sum_{k=0}^{\infty} \frac{x^{2k}}{(2k+1)!} + 1 - \sum_{k=1}^{\infty}\frac{2^{2k} \vert B_{2k} \vert}{(2k)!} x^{2k} \\
&= 2 + \sum_{k=1}^{\infty} \frac{1}{(2k)!}\left[\frac{1}{(2k+1)} - 2^{2k} \vert B_{2k} \vert \right] x^{2k} \\
&= 2 + \left( \sum_{k=1}^{\infty} b_k x^{2k} := B(x) \right) \ \ (0 < x < \pi/2),
\end{align*}
where $ b_k = \frac{1}{(2k)!}\left[\frac{1}{(2k+1)} - 2^{2k} \vert B_{2k} \vert \right] < 0 $ for all $ k \geq 1 $ by virtue of Lemma \ref{lem6}.
So $$ \frac{\sinh x}{x} + \left(\frac{x}{\tan x}\right)^2 < 2 + b_1 x^2 \ \ (0 < x < \pi/2) $$ and by $ b_1 = -1/6 $ we get the desired leftmost  inequality of (\ref{eqn3.3}).

To prove the rightmost double inequality of (\ref{eqn3.3}), we equivalently  prove that
$$ \tan x \sinh^2 x - 2 x^2 \tan x + x^3 > \frac{x^7}{45} > 0 \ \ (0 < x < \pi/2). $$
Let
$$ f(x) = \tan x \sinh^2 x - 2 x^2 \tan x + x^3 \ \  (0 < x < \pi/2). $$
It can be written as
$$ f(x) = x^3 - 2 x^2 \tan x + \frac{\tan x}{2} \cosh 2x - \frac{\tan x}{2} . $$
Using known series expansions of $ \cosh x $ (see (\ref{eqn2.1})) and $ \tan x $ (see (\ref{eqn2.6})), we write
\begin{align*}
f(x) &= x^3 - x^2 \tan x + \sum_{k=2}^{\infty} \frac{2^{2k-1}}{(2k)!}  x^{2k} \tan x \\
&= x^3 - x^2 \left( x + \frac{1}{3}x^3 + \frac{2}{15}x^5 + \frac{17}{315}x^7 + \frac{62}{2835}x^9 + \cdots \right) \\
&+ \frac{1}{3}x^4 \tan x + \frac{2}{45}x^6 \tan x + \frac{1}{315}x^8 \tan x + \frac{2}{14175}x^{10} \tan x + \cdots \\
&= x^3 + \left(-x^3 - \frac{1}{3}x^5 - \frac{2}{15}x^7 - \frac{17}{315}x^9 - \frac{62}{2835}x^{11} - \cdots \right) \\
&+ \left( \frac{1}{3}x^4 + \frac{2}{45}x^6 + \frac{1}{315}x^8 + \frac{2}{14175}x^{10} + \cdots \right) \times \\
 &\left(x + \frac{1}{3}x^3 + \frac{2}{15}x^5 + \frac{17}{315}x^7 + \frac{62}{2835} x^9 + \cdots \right) \\
 &= \left( - \frac{1}{3}x^5 - \frac{2}{15}x^7 - \frac{17}{315}x^9 - \frac{62}{2835}x^{11} - \cdots \right) \\
 &+ \left( \frac{1}{3}x^5 + \frac{7}{45} x^7 + \frac{59}{945}x^9 + \frac{2492}{99225}x^{11} + \cdots \right) \\
 &= \frac{1}{45}x^7 + \frac{8}{945} x^9 + \frac{46}{14175} x^{11} + \cdots   > \frac{1}{45}x^7 > 0.
\end{align*}
The proof is completed.
$\hfill\square$

\begin{remark}
The rightmost inequality of (\ref{eqn3.3}) can be put as
\begin{align}\nonumber
\left(\frac{\sinh x}{x}\right)^2 + \frac{x}{\tan x} \left(1-\frac{x^4}{45}\right) > 2 \ \ (0 < x < \pi/2).
\end{align}
\end{remark}

\begin{corollary}\label{cor1}
For $ 0 < x < \pi/2, $ we have the following chains of inequalities:
\begin{align}\label{eqn3.5}
\left(\frac{x}{\sin x}\right)^2 + \frac{\tanh x}{x} > \left(\frac{\sinh x}{x}\right)^2 + \frac{\tanh x}{x} > \left(\frac{\sinh x}{x}\right)^2 + \frac{x}{\tan x} > 2 + \frac{x^5}{45 \tan x} > 2
\end{align}
and
\begin{align}\label{eqn3.6}
\left(\frac{x}{\sinh x}\right)^2 + \frac{\tan x}{x} > \left(\frac{\sin x}{x}\right)^2 + \frac{\tan x}{x} > \left(\frac{\sin x}{x}\right)^2 + \frac{x}{\tanh x} > 2.
\end{align}
\end{corollary}
\begin{proof}

The chain of inequalities in (\ref{eqn3.5}) immediately follows by applying Theorem \ref{thm2} with Lemmas \ref{lem1} and \ref{lem3}. Similarly (\ref{eqn3.6}) follows due to Theorem \ref{thm1} and Lemmas \ref{lem1} and \ref{lem3}.
\end{proof}
Again the following corollary is an easy consequence of Theorems \ref{thm1} and \ref{thm2}.
\begin{corollary}\label{cor2}
For $ 0 < x < \pi/2 , $ we have following chain of inequalities:
\begin{align}\label{eqn3.7}
\left(\frac{x}{\sin x}\right)^2 + \frac{x}{\tan x} > \left(\frac{\sinh x}{x}\right)^2 + \frac{x}{\tan x} > 2 + \frac{x^5}{45 \tan x} > 2
\end{align}
and
\begin{align}\label{eqn3.8}
\left(\frac{x}{\sinh x}\right)^2 + \frac{x}{\tanh x} > \left(\frac{\sin x}{x}\right)^2 + \frac{x}{\tanh x}  > 2.
\end{align}
\end{corollary}

As the applications of Theorems \ref{thm1} and \ref{thm2}, we have in fact obtained/refined the inequalities (\ref{eqn1.1})-(\ref{eqn1.4}) in Corollaries \ref{cor1} and \ref{cor2}. 

\vspace{.5cm}

\noindent{\bf Proof of Theorem \ref{thm3}.}
We first prove the inequalities in (\ref{eqn3.9}). For the leftmost inequality, we have by (\ref{eqn2.2}) and (\ref{eqn2.4}) that
\begin{align*}
2 \frac{\sinh x}{x} + \frac{x}{\tan x}  &= \sum_{k=0}^{\infty} \frac{2}{(2k+1)!}x^{2k} + 1 - \sum_{k=1}^{\infty} \frac{2^{2k} \vert B_{2k} \vert}{(2k)!} x^{2k} \\
&= 3 + \sum_{k=1}^{\infty} \frac{2}{(2k+1)!}x^{2k} - \sum_{k=1}^{\infty}\frac{2^{2k} \vert B_{2k} \vert}{(2k)!} x^{2k} \\
&= 3 + \sum_{k=2}^{\infty} \frac{2}{(2k)!}\left(\frac{1}{2k+1} - 2^{2k-1} \vert B_{2k} \vert \right)x^{2k} \\
&< 3 - \frac{x^4}{180} < 3
\end{align*}
by Lemma \ref{lem5}.

For the rightmost inequality of (\ref{eqn3.9}) we use (\ref{eqn2.31}) and (\ref{eqn2.6}) to get
\begin{align*}
2 \frac{x}{\sinh x} + \frac{\tan x}{x} &=   2 - \sum_{k=1}^{\infty} \frac{2^2(2^{2k-1}-1)}{(2k)!} B_{2k} x^{2k} + \sum_{k=1}^{\infty} \frac{2^{2k}(2^{2k}-1)}{(2k)!}\vert B_{2k} \vert x^{2k-2} \\
&= 2 - \sum_{k=1}^{\infty} \frac{2^2(2^{2k-1}-1)}{(2k)!} B_{2k} x^{2k} + \sum_{k=0}^{\infty} \frac{2^{2k+2}(2^{2k+2}-1)}{(2k+2)!}\vert B_{2k+2} \vert x^{2k} \\
&= 3 + \sum_{k=2}^{\infty} \frac{2^2}{(2k)!}\left[\frac{2^{2k-1}(2^{2k+2}-1)}{(k+1)(2k+1)}\vert B_{2k+2} \vert - (-1)^{k+1}(2^{2k-1}-1)\vert B_{2k} \vert \right]x^{2k} \\
&= 3 + \left( \sum_{k=2}^{\infty} c_k x^{2k} := C(x) \right),
\end{align*}
where $ c_k = \frac{2^2}{(2k)!}\left[\frac{2^{2k-1}(2^{2k+2}-1)}{(k+1)(2k+1)}\vert B_{2k+2} \vert - (-1)^{k+1}(2^{2k-1}-1)\vert B_{2k} \vert \right] \ \ (k \geq 2). $ Now $ c_k > 0 $ for $ k = 2, 4, 6, \cdots $ and  $$ c_k = \frac{2^2}{(2k)!}\left[\frac{2^{2k-1}(2^{2k+2}-1)}{(k+1)(2k+1)}\vert B_{2k+2} \vert - (2^{2k-1}-1)\vert B_{2k} \vert \right] > 0 \ \ \text{for} \ \ k = 3, 5, 7, \cdots $$ by Lemma \ref{lem7}. Thus all the terms of $ C(x) $ are positive and hence by truncating $ C(x) $ we get 
 $$ 2 \frac{x}{\sinh x} + \frac{\tan x}{x} > 3 + c_2 x^4 > 3 \ \ \text{where} \ \ c_2 = \frac{31}{180}. $$ 
 The inequalities (\ref{eqn3.10}) can be proved easily by applying similar arguments. 
$\hfill\square$

\vspace{.5cm}

\noindent{\bf Proof of Theorem \ref{thm4}.}
We give proof of the inequalities in (\ref{eqn3.11}) only. The inequalities in (\ref{eqn3.12}) will follow similarly. 
From (\ref{eqn2.2}) and (\ref{eqn2.5}) we have
\begin{align*}
2 \frac{\sin x}{x} + \frac{x}{\tanh x} &= 2 \sum_{k=0}^{\infty} \frac{(-1)^k}{(2k+1)!}x^{2k} + 1 + \sum_{k=1}^{\infty}\frac{2^{2k}  B_{2k} }{(2k)!} x^{2k} \\
&= 3 + \sum_{k=1}^{\infty} \frac{2(-1)^k}{(2k+1)!}x^{2k} + \sum_{k=1}^{\infty}\frac{2^{2k} \vert B_{2k} \vert}{(2k)!}x^{2k} \\
&= 3 + \sum_{k=1}^{\infty}\frac{2}{(2k)!}\left[\frac{(-1)^k}{2k+1} - (-1)^k 2^{2k-1} \vert B_{2k} \vert \right]x^{2k} \\
&= 3 +  \sum_{k=2}^{\infty}\frac{2(-1)^k}{(2k)!}\left[\frac{1}{2k+1} - 2^{2k-1} \vert B_{2k} \vert \right]x^{2k} \\
&= 3 + \left( \sum_{k=2}^{\infty}d_k x^{2k} := D(x) \right),
\end{align*}
where $ d_k = \frac{2(-1)^k}{(2k)!}\left[\frac{1}{2k+1} - 2^{2k-1} \vert B_{2k} \vert \right] \ (k \geq 2). $ By Lemma \ref{lem5}, the series $ D(x) $ is an alternating convergent series and its first term is negative. Consequently $$ 3 + d_2 x^4 < 2 \frac{\sin x}{x} + \frac{x}{\tanh x} < 3 $$ where $ d_2 = -\frac{1}{180}. $ This gives leftmost double inequality of (\ref{eqn3.11}).

For the rightmost double inequality of (\ref{eqn3.11}), we make use of (\ref{eqn2.3}) and (\ref{eqn2.7}) to obtain
\begin{align*}
2 \frac{x}{\sin x} + \frac{\tanh x}{x} &= 2 + \sum_{k=1}^{\infty} \frac{2^2(2^{2k-1}-1)}{(2k)!}\vert B_{2k} \vert x^{2k} + \sum_{k=1}^{\infty} \frac{2^{2k}(2^{2k}-1)}{(2k)!} B_{2k} x^{2k-2} \\
&= 2 + \sum_{k=1}^{\infty} \frac{2^2(2^{2k-1}-1)}{(2k)!}\vert B_{2k} \vert x^{2k} + \sum_{k=1}^{\infty} \frac{2^{2k+2}(2^{2k+2}-1)}{(2k+2)!} B_{2k+2} x^{2k} \\
&= 3 + \sum_{k=2}^{\infty} \frac{2^2}{(2k)!} \left[ (2^{2k-1} - 1) \vert B_{2k} \vert + (-1)^k \frac{2^{2k-1}(2^{2k+2}-1)}{(k+1)(2k+1)}\vert B_{2k+2} \vert \right]x^{2k} \\
&= 3 + \left( \sum_{k=2}^{\infty} e_k x^{2k} = E(x) \right),
\end{align*}
where $ e_k = \frac{2^2}{(2k)!} \left[ (2^{2k-1} - 1) \vert B_{2k} \vert + (-1)^k \frac{2^{2k-1}(2^{2k+2}-1)}{(k+1)(2k+1)}\vert B_{2k+2} \vert \right] \ (k \geq 2). $ Clearly $ e_k > 0 $ for $ k = 2, 4, 6, \cdots $ and $ e_k < 0 $ for $ k = 3, 5, 7, \cdots $ by Lemma \ref{lem7}. Thus $ E(x) $ is an alternating convergent series whose first term is positive, i.e. $ \frac{31}{180}x^4.$ The desired inequality follows by the truncation of $ E(x). $ This ends the proof.
$\hfill\square$

\begin{remark}
If $ n $ is even positive integer and $ 0 < x < \pi/2 $ then
\begin{align}\nonumber
3 + \sum_{k=2}^{n} d_k x^{2k} < 2 \frac{\sin x}{x} + \frac{x}{\tanh x} < 3 + \sum_{k=2}^{n+1} d_k x^{2k},
\end{align}
and
\begin{align}\nonumber
3 + \sum_{k=2}^{n+1}e_k x^{2k} < 2 \frac{x}{\sin x} + \frac{\tanh x}{x} < 3 + \sum_{k=2}^{n}e_k x^{2k},
\end{align}
where $ d_k $ and $ e_k $ are as defined in the proof of Theorem \ref{thm4}.
\end{remark}

Before concluding this section, we must emphasize that the Huygens-type inequalities in Theorems \ref{thm3} and \ref{thm4} give sharp bounds to the much-discussed functions in the literature such as $ \sin(x)/x, x/\tan(x) $ etc. For example, the inequalities
\begin{align}\label{eqn3.15}
\frac{\sin x}{x} < \frac{1}{2}\left(3 - \frac{x}{\tanh x}\right) \ \ (0 < x < \pi/2),
\end{align}
and
\begin{align}\label{eqn3.16}
\frac{\sinh x}{x} < \frac{1}{2}\left(3 - \frac{x}{\tan x}\right) \ \ (0 < x < \pi/2)
\end{align}
are very sharp and interesting for further studies.

\end{document}